\documentclass[12pt]{amsart}
\usepackage{amscd,amsfonts,amsmath,amssymb,amsthm,amstext,latexsym}
\usepackage{enumerate,pifont,graphicx}
\usepackage[english]{babel}
\usepackage[T1]{fontenc}
\usepackage{epsfig}

\def\a{{\alpha}}

\def\g{{\gamma}}
\def\b{{\beta}}

\def\O{{\Omega}}

\newcommand{\R}{{\mathbb R}}

\newcommand{\N}{{\mathbb N}}

\newcommand{\htwo}{{\mathbb H^2}}
\newcommand{\htwor}{{\mathbb H^2}\times {\mathbb R} \ }

\def\t{{\theta}}

\def\G{{\Gamma}}

\newtheorem{theorem}{Theorem}[section]
\newtheorem{lemma}[theorem]{Lemma}
\newtheorem{proposition}[theorem]{Proposition}
\newtheorem{remark}[theorem]{Remark}

\newtheorem{definition}[theorem]{Definition}
\begin{document}
\title{Saddle Towers in $\htwo\times\R$}
\author{Filippo Morabito}
\author{M. Magdalena Rodr\'\i guez}
\thanks{The first author is supported by a grant of Instituto de
  Matem\'atica Interdisciplinar of Universidad Complutense de Madrid.
  The second author is partially supported by a Spanish MEC-FEDER
  Grant no. MTM2007-61775 and a Regional J. Andaluc\'\i a Grant no.
  P06-FQM-01642.}
\address{Instituto de Matem\'atica Interdisciplinar, Universidad
  Complutense de Madrid, Plaza de las Ciencias 3, 28040, Madrid,
  Spain} \email{morabitf@gmail.com}
\address{Departamento de \'Algebra, Universidad Complutense de Madrid,
  Plaza de las Ciencias~3, 28040, Madrid, Spain}
\email{magdalena@mat.ucm.es}

\begin{abstract} 
  Given $k\geqslant 2$, we construct a $(2k-2)$-parameter family of
  properly embedded minimal surfaces in $\htwo\times\R$ invariant by a
  vertical translation $T$, called {\it Saddle Towers}, which have
  total intrinsic curvature $4 \pi(1-k),$ genus zero and $2k$ vertical
  Scherk-type ends in the quotient by $T$. As limits of those Saddle
  Towers, we obtain Jenkins-Serrin graphs over ideal polygonal domains
  (with total intrinsic curvature $2\pi(1-k)$); we also get properly
  embedded minimal surfaces which are symmetric with respect to a
  horizontal slice and have total intrinsic curvature $4 \pi(1-k)$,
  genus zero and $k$ vertical planar ends.
\end{abstract}
\maketitle

\section{Introduction}

H. F. Scherk~\cite{S} found a singly periodic minimal surface 
in $\R^3$
invariant by a vertical translation, which can be seen as the
desingularization of two orthogonal vertical planes. This is the
conjugate surface of the doubly periodic minimal surface obtained from
the graph surface of 
\[
u(x,y)=\log\frac{\cos x}{\cos y},\quad |x|<\frac\pi 2,
\quad |y|<\frac\pi 2 \, ,
\]
rotating it by an angle $\pi$ about the straight vertical lines
in its boundary.  Such singly periodic minimal surface can be seen in
a 1-parameter family of singly periodic minimal surfaces invariant by
a vertical translation, by changing the angle between the vertical
planes. They are called {\it singly periodic Scherk minimal examples}.

\medskip

In general, consider a convex polygonal domain $\Omega\subset\R^2$
with $2k$ edges of length one, with $k\geqslant 2$.  Mark its edges
alternately by $+\infty$ and $-\infty$.  H. Jenkins and
J. Serrin~\cite{JS} gave necessary and sufficient conditions for the
existence of a function $u$ defined on $\Omega$ which goes to $\pm
\infty$ on the edges, as indicated by the marking. To satisfy such
conditions, $\Omega$ is assumed to be different from a parallelogram
bounded by two sides of length one and two sides of length $k-1$, for
$k\geqslant 3$ (see for instance Mazet, Rodríguez and
Traizet~\cite[Proposition 1.3]{MRT}).

The graph surface $\Sigma_u$ of $u$ is bounded by $2k$ vertical
straight lines above the vertices of $\Omega$.  The conjugate minimal
surface of $\Sigma_u$ is then bounded by $2k$ horizontal symmetry
curves, lying in two horizontal planes at distance one from each
other.  By reflecting about one of the two symmetry planes, we obtain
a fundamental domain for a properly embedded singly periodic minimal
surface $M$ of period $T=(0,0,2)$. In the quotient by~$T$, $M$ has
genus zero and $2k$ ends asymptotic to flat vertical annuli (quotients
of vertical half-planes by $T$). This kind of ends are classically
called Scherk-type ends. Remark that changing the length $\ell$ of the
edges of $\Omega$ gives nothing but $M$ rescaled by $\ell$. This is
why we can fix $\ell=1$.

This procedure provides for $k=2$ the $1$-parameter family of Scherk
examples; and for any $k\geqslant 3$, a $(2k-3)$-parameter family of
examples, which were constructed by H.~Karcher~\cite{K1,K2} and called
{\it Saddle Towers}. These examples have recently been classified by
J. P\'erez and M. Traizet~\cite{PT} as the only complete embedded
singly periodic minimal surfaces in~$\R^3$ with genus zero and
finitely many Scherk-type ends in the quotient.

\medskip

In this paper we follow the same strategy in $\htwo\times\R$ to
construct properly embedded singly periodic minimal surfaces in
$\htwo\times\R$ invariant by a vertical translation $T$, with genus
zero and $2k$ vertical Scherk-type ends in the quotient by $T$. We say
that an end is a vertical Scherk-type end when it is asymptotic to the
quotient by $T$ of half a vertical geodesic plane.

\begin{theorem}\label{th:main}
  Given $k\geqslant 2$ and a vertical translation $T$, there exists a
  $(2k-3)$-parameter family of properly embedded singly periodic
  minimal surfaces in $\htwo\times\R$ with total (intrinsic) curvature
  $4\pi(1-k)$, genus zero and $2k$ vertical Scherk-type ends in the
  quotient by $T$. We call them {\it Saddle Towers}.
\end{theorem}

Independently, H. Lee and J. Pyo~\cite{LP} have recently constructed
symmetric Saddle Towers following two different approaches: the
conjugation method explained above and a barrier method (see
Remark~\ref{rem:symmetry}).

\medskip

We observe that we do not have homotheties in $\htwo\times\R$, so the
length of $T$ gives us another parameter of the family. Then, we
obtain a $(2k-2)$-parameter family of Saddle Towers.  The following
theorems gives us possible limits of Saddle Towers when the length of
$T$ goes to $+\infty$.

\begin{theorem}\label{th:limits}
  Given $k\geqslant 2$, there exists a $(2k-3)$-parameter family of
  properly embedded minimal surfaces in $\htwo\times\R$ with total
  (intrinsic) curvature $4\pi(1-k)$, genus zero and $k$ ends, each one
  asymptotic to a vertical geodesic plane. Those surfaces are
  invariant by the reflection symmetry about $\htwo\times\{0\}$, and
  can be obtained as limits of Saddle Towers invariant by
  $T=(0,0,2\ell)$, with $\ell$ diverging to $+\infty$.
\end{theorem}

Theorem~\ref{th:limits} includes the 1-parameter family that
J. Pyo~\cite{P} has independently constructed very recently. These are
the examples explained in Remark~\ref{rem:symmetrySI}.

Note that Theorem~\ref{th:limits} provides the existence of properly
embedded minimal surfaces with genus zero and $k\geqslant 3$ ends in
$\htwo\times\R$, which is not possible when the ambient space is
$\R^3$.  This shows again that $\htwor$ allows for a wider variety of
examples than $\R^3$.

L. Hauswirth and H. Rosenberg prove in~\cite{HR} that, when it is
finite, the total curvature of a complete embedded minimal surface in
$\htwor$ is a multiple of $2\pi$, and give simply connected examples
whose total curvature is $-2\pi m$, for each non-negative integer $m$.
They also suggest that it would be interesting to construct non-simply
connected examples of finite total curvature; for example an annulus
of total curvature $-4\pi$.  Theorem~\ref{th:limits} includes such
examples.

The examples constructed by Hauswirth and Rosenberg are minimal graphs
over polygonal domains with $2(m+1)$ edges, whose vertices are at
infinity (called ideal polygonal domains), with boundary values
$\pm\infty$ alternately. We call these graphs {\it Jenkins-Serrin
  graphs over ideal polygonal domains}. The Jenkins-Serrin problem
over ideal polygonal domains was studied by P. Collin and H. Rosenberg
in~\cite{CR}.  In Section~\ref{ideal} we prove that the Jenkins-Serrin
graphs over ideal polygonal domains can be obtained by taking limits
of half of the complete surfaces satisfying Theorem~\ref{th:limits} or
taking limits of Saddle Towers.

\begin{theorem}\label{th:limitsJS}
  Given $k\geqslant 2$, let $\mathcal{F}_k$ be the $(2k-3)$-parameter
  family of Jenkins-Serrin graphs over an ideal polygonal domains with
  $2k$ edges. Then $\mathcal{F}_k$ is self-conjugate in the sense that
  the conjugate surface of a graph in $\mathcal{F}_k$ also belongs
  to~$\mathcal{F}_k$. Furthermore, the graphs in $\mathcal{F}_k$ can
  be obtained as limits of Saddle Towers with $2k$ ends, invariant by
  $T=(0,0,2\ell)$, when $\ell$ diverges to $+\infty$.
\end{theorem}

\section{Preliminaries}
\label{sec:pre}
All surfaces in the paper are supposed to be connected and orientable.

We consider the Poincar\'{e} disk model of $\htwo$,
\[
\htwo=\{(x,y)\in\R^2\ |\ x^2+y^2<1\},
\]
with the hyperbolic metric
$g_{-1}=\frac{4}{(1-x^2-y^2)^2}(dx^2+dy^2)$, and denote by $t$ the
coordinate in~$\R$. Consider in $\htwo\times\R$ the usual
product metric
\[
ds^2=\frac{4}{(1-x^2-y^2)^2}(dx^2+dy^2)+dt^2.
\]

\subsection{Minimal graphs in $\htwo\times\R$}\label{subsec}
Let $\Omega \subset {\mathbb H}^2 $ be an open domain and $u:\Omega
\to \R$ a smooth function.  The (vertical) graph of $u$ is minimal in
$\htwor$ if
\begin{equation}
\label{eq.min.surf}
{\rm div}\left( \frac {\nabla u}
{\sqrt{1+|\nabla u|^2}} \right)=0,
\end{equation}  
where all terms are calculated in the metric of $\htwo
$.

In~\cite{NR}, B. Nelli and H. Rosenberg proved a Jenkins-Serrin type
theorem for simply connected bounded convex domains in ${\mathbb H}^2
\times \R$: Let $\Omega\subset\htwo$ be a simply connected bounded
convex domain whose boundary consists of a finite number of geodesic
arcs $A_1,\ldots,A_n$, $B_1,\ldots, B_m$ and a finite number of convex
arcs $C_1,\ldots,C_p$ (convex with respect to $\Omega$), together with
their endpoints, such that no two $A_i$ edges and no two $B_i$ edges
have a common endpoint. They gave necessary and sufficient conditions
(in terms of the lengths of the boundary arcs of $\Omega$ and of the
perimeter of inscribed polygons in $\Omega$ whose vertices are among
the vertices of $\Omega$) for the existence and uniqueness (up to an
additive constant, in the case the family of $C_i$ arcs is empty) of a
solution $u$ for the minimal graph equation~\eqref{eq.min.surf} such
that
\[
u_{|A_i}=+\infty,\quad  u_{|B_i}=-\infty, \quad\mbox{and}\quad
u_{|C_i}=f_i,
\]
for arbitrary continuous functions $f_i$.

P. Collin and H. Rosenberg~\cite{CR} solved the Jenkins-Serrin problem
for unbounded simply-connected domains bounded by a finite number of
complete geodesic arcs and a finite number of complete convex arcs,
together with their endpoints at $\partial_\infty\htwo$, with the
additional assumption that two consecutive boundary edges of $\Omega$
are asymptotic at their common endpoint at $\partial_\infty\htwo$.  A
general Jenkins-Serrin problem in $\htwo\times\R$ was solved by the
second author together with L. Mazet and H. Rosenberg in~\cite{MRR}.

In this work we will consider the particular case where $\Omega$ is a
convex polygonal domain with $2k$ geodesic edges
$A_1,B_1,\ldots,A_k,B_k$ (cyclically ordered) of the same length
$\ell\in(0,+\infty]$. When $\ell=+\infty$, $\Omega$ is assumed to be
either an ideal or a semi-ideal polygonal domain, see
Definition~\ref{def:ideal} below.
We will state the Jenkins-Serrin theorem for such domain
$\Omega$. Before, we fix some notation.

\begin{definition}\label{def:ideal}
  {\rm Let $\Omega$ be a polygonal domain (i.e. a domain whose edges
    are geodesic arcs). We call {\it ideal vertices} of $\Omega$ to
    those of its vertices that are at $\partial_\infty\htwo$. We say
    that $\Omega$ is {\it ideal} when all its vertices are at
    $\partial_\infty\htwo$.  We say that $\Omega$ is {\it semi-ideal}
    when it has an even number of vertices $p_1,\ldots,p_{2k}$
    (cyclically ordered), such that the odd vertices $p_{2i-1}$ are in
    $\htwo$ and the even vertices $p_{2i}$ are at infinity
    $\partial_\infty\htwo$ (or viceversa).  }
\end{definition}

For each ideal vertex $p_i$ of $\Omega$ (if it exists), we consider a
horocycle $H_i$ at $p_i$.  Assume $H_i$ is small enough so that it
only intersects $\partial\Omega$ at the boundary edges having $p_i$ as
an endpoint, and so that $H_i\cap H_j=\emptyset$, for every $i\neq j$.
Given a polygonal domain $\mathcal{P}$ inscribed in $\Omega$ (i.e. a
polygonal domain $\mathcal{P}\subset\Omega$ whose vertices are drawn
from the set of endpoints of the $A_i,B_i$ edges, possibly at
infinity), we denote by $\Gamma(\mathcal{P})$ the part of $\partial
\mathcal{P}$ outside the horocycles (observe that
$\Gamma(\mathcal{P})=\partial \mathcal{P}$ in the case $\Omega$ has no
ideal vertices). Also let us call
\[
\a(\mathcal{P})=\sum_i\left|A_i\cap\Gamma(\mathcal{P})\right|
\qquad {\rm and}\qquad
\b(\mathcal{P})=\sum_i\left|B_i\cap\Gamma(\mathcal{P})\right|
,
\]
where $|\cdot|=\mbox{length}_\htwo(\cdot)$. 

\begin{definition}
  \label{Def:JS}
  {\rm Let $\Omega$ be a convex polygonal domain as above. We say that
    $\Omega$ is a {\it Jenkins-Serrin domain} when the following two
    additional conditions hold for some choice of horocycles at its
    ideal vertices:
    \begin{itemize}
    \item[(i)] $\a(\Omega)=\b(\Omega)$.
    \item[(ii)]
      $2\a(\mathcal{P})<\left|\Gamma(\mathcal{P})\right|$
      and
      $2\b(\mathcal{P})<\left|\Gamma(\mathcal{P})\right|$
      for every polygonal domain $\mathcal{P}$ inscribed in~$\Omega$,
      $\mathcal{P}\neq\Omega$.
  \end{itemize}
}
\end{definition}
Remark that condition (i) in the above definition does not depend on
the choice of horocycles $H_i$; and if the inequalities of condition
(ii) are satisfied for some choice of horocycles, then they continue
to hold for ``smaller'' horocycles.

\begin{theorem} \label{th:JS} Let $\Omega$ be a convex polygonal
  domain with $2k$ edges $A_1,B_1,\ldots,A_k,B_k$ (cyclically ordered)
  of the same length $\ell\in(0,+\infty]$.  There exists a solution
  $u$ for the minimal graph equation~\eqref{eq.min.surf} such that
  \[
  u_{|A_i}=+\infty\quad\mbox{and}\quad u_{|B_i}=-\infty,
  \]
  if, and only if, $\Omega$ is a Jenkins-Serrin domain. Moreover, if
  it exists, then it is unique up to an additive constant.
\end{theorem}

\begin{figure}
  \begin{center}
    \epsfysize=6.5cm \epsffile{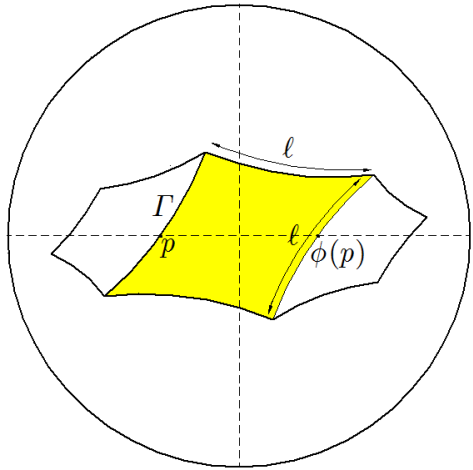}
  \end{center}
  \caption{Consider the translation $\phi$ in the direction of the
    $x$-axis such that the edges of the geodesic square $\mathcal{D}$
    determined by $\Gamma$ and $\phi(\Gamma)$ (the shadowed region)
    have length $\ell$. Translating twice $\mathcal{D}$ by $\phi$ we
    get a polygonal domain of eight edges of length $\ell$ which is
    not a Jenkins-Serrin domain.}
  \label{Fig1}
\end{figure}

\begin{remark}
  In $\R^3$, the only convex polygonal domains with $2k$ edges of the
  same length $\ell\in(0,+\infty)$ which are not Jenkins-Serrin
  domains are parallelograms bounded by two sides of length $\ell$ and
  two sides of length $(k-1)\ell$, with $k\geq 3$
  (see~\cite[Proposition 1.3]{MRT}). In $\htwo\times\R$, it is not so
  restrictive. For instance:
  \begin{itemize}
  \item Let $\G$ be a geodesic arc of length $\ell$, and let $p$ be
    its middle point. Consider a geodesic $\gamma$ passing through
    $p$, and a hyperbolic translation~$\phi$ along~$\g$ such that the
    distance from the endpoints of~$\G$ and $\phi(\G)$ is~$\ell$. Call
    $\mathcal{D}$ the polygonal domain of four edges determined by
    $\G,\phi(\G)$. The convex polygonal domain obtained from
    $\mathcal{D}$ by translating it $k$ times by
    $\phi$ (see Figure~\ref{Fig1}).
    is a polygonal domain of $4+2k$ edges of length $\ell$ which is
    not a Jenkins-Serrin domain.
  \item It can be also considered a convex polygonal domain
    $\mathcal{D}$ of $2n$ edges of length $\ell$ such that the
    interior angles at their vertices are smaller than or equal to
    $\pi/2$.  By reflecting $k$ times $\mathcal{D}$ about two opposite
    edges, we obtain a convex polygonal domain with $2n+2(n-1)k$ edges
    of length $\ell$ which is not a Jenkins-Serrin domain (see
    Figure~\ref{Fig2}).
  \end{itemize}
\end{remark}

\begin{figure}
  \begin{center}
    \epsfysize=6.5cm \epsffile{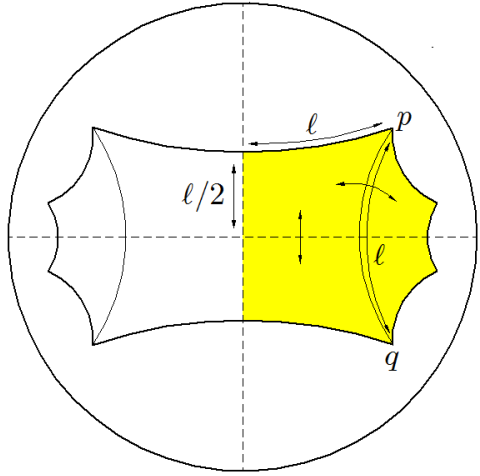}
  \end{center}
  \caption{The shadowed region $\mathcal{D}$ is a geodesic hexagon
    whose edges have length $\ell$ and their interior angles are
    $\pi/2$ up to at two opposite vertices $p,q$ where the interior
    angles are strictly smaller than $\pi/2$. By reflecting
    $\mathcal{D}$ with respect to one of the edges who has not $p$ nor
    $q$ as an endpoint, we get a polygonal domain of ten edges of
    length $\ell$ which is not a Jenkins-Serrin domain.}
  \label{Fig2}
\end{figure}

\subsection{Conjugate surfaces in $\htwo\times\R$}
\label{conjugate.h2}
In this subsection we will recall how to obtain minimal surfaces in
$\htwo\times\R$ by conjugation from other known minimal examples.  For
more details see B. Daniel~\cite{D}.

Let $\Sigma$ be a simply connected Riemann surface and $J$ be the
rotation of angle~$\frac{\pi}{2}$ on~$T\Sigma$.  Denote by
$\langle\cdot,\cdot\rangle$ the Riemannian metric on $\Sigma$. Given a
conformal minimal immersion $X:\Sigma\to\htwo\times\R$, let us call:
\begin{itemize}
\item $S$ the symmetric operator on $\Sigma$ induced by the shape
  operator of $X(\Sigma)$;
\item $T$ the vector field such that $dX(T)$ is the projection of
  $\frac{\partial}{\partial t}$ onto $T(X(\Sigma))$;
\item $N$ the induced unit normal field on $X(\Sigma)$;
\item $\nu=\langle N,\frac{\partial}{\partial t} \rangle$ the angle
  function (in particular, $\|T\|^2+\nu^2=1$).
\end{itemize}

\begin{theorem}[{\cite[Theorem 4.2]{D}}]\label{th:conjugate}
  Let $X:\Sigma\to\htwo\times\R$ be a conformal minimal immersion and
  $z_0 \in \Sigma$. There exists a unique conformal minimal immersion
  $X^*:\Sigma\to\htwo\times\R$ such that:
\begin{enumerate}
\item $X^*(z_0)=X(z_0)$ and $(dX^*)_{z_0}=(dX)_{z_0}$;
\item the metrics induced on $\Sigma$ by $X$ and $X^*$ are the same;
\item the symmetric operator on $\Sigma$ induced by the shape operator
  of $X^*(\Sigma)$ is $S^*=JS$;
\item $\frac{\partial}{\partial t}=dX^*(T^*)+\nu N^*$, where $T^*=JT$
  and $N^*$ is the unit normal vector to $X^*.$
\end{enumerate}
\end{theorem}

\begin{definition}
  {\rm The immersion $X^*$ obtained in Theorem~\ref{th:conjugate} is
    usually called the {\it conjugate immersion} of~$X$, and
    $X^*(\Sigma)$ is the {\it conjugate surface} of $X(\Sigma)$. In
    fact, we will not assume condition {\it (1)} in
    Theorem~\ref{th:conjugate} for the definition of conjugate
    surface; i.e. we will consider the conjugate surface $X^*(\Sigma)$
    well-defined up to isometries of $\htwo\times\R$ preserving the
    orientation of both $\htwo$ and $\R$.}
\end{definition}

Let $X=(\varphi,h):\Sigma\to\htwo\times\R$ be a conformal minimal
immersion. Since $X$ is minimal, the height function $h$ is a real
harmonic function, and that $\varphi$ is a harmonic map to $\htwo$. In
particular, we can define the Hopf differential of $\varphi$, given by
$Q_\varphi:=4 \left\|\frac{\partial \varphi}{\partial
    z}\right\|^2\,dz^2$.  Since $X$ is conformal, $Q_\varphi:=-4
\left(\frac{\partial h}{\partial z}\right)^2\,dz^2$, where $z=u+iv$ is
a local coordinate on $\Sigma$. 

\begin{proposition}[{\cite[Proposition 4.6]{D}}]\label{prop:conjugate}
  Let $X=(\varphi,h):\Sigma\to\htwo\times\R$ be a conformal minimal
  immersion, and denote by
  $X^*=(\varphi^*,h^*):\Sigma\to\htwo\times\R$ its conjugate
  immersion. Then, $h^*$ is the (real) harmonic conjugate of $h$ and
  the Hopf differential of $\varphi^*$ is
  ${Q_{\varphi^*}=-Q_\varphi}$.
\end{proposition}

We remark that L. Hauswirth, R. Sa Earp and 
E. Toubiana~\cite{HST} gave a
different definition of conjugate immersion: Given a conformal minimal
immersion ${X=(\varphi,h):\Sigma\to\htwo\times\R}$, they proved there
exists a conformal minimal immersion
${X^*=(\varphi^*,h^*):\Sigma\to\htwo\times\R}$ isometric to $X$ such
that ${Q_{\varphi^*}=-Q_\varphi}$. They call conjugate immersion of
$X$ to such conformal minimal immersion $X^*$. Such conjugate
immersion is well-defined up to an isometry of $\htwo\times\R$ (not
necessarily preserving the orientation of both $\htwo$ and $\R$), by
Theorem~\ref{th:isometric} below. We will use in this paper Daniel's
definition of conjugate immersion.

\begin{theorem}[{\cite[Theorem 6]{HST}}]\label{th:isometric}
  Two isometric conformal minimal immersions
  $X_1=(\varphi_1,h_1),X_2=(\varphi_2,h_2)$ of $\Sigma$ in $\htwo\times\R$
  having the same Hopf differential $Q_{\varphi_1}=Q_{\varphi_2}$, are
  equal up to an isometry of $\htwo\times\R$.
\end{theorem}

Some geometric properties of conjugate surfaces in $\htwo\times\R$ are
also discussed in~\cite{D}.  Similarly as in $\R^3$ (see
Karcher~\cite{K1,K2,K3}), we get the following lemma.

\begin{lemma}\label{lem:Schwarz}
  The conjugation exchanges the following Schwarz reflections:
\begin{itemize}
\item The symmetry with respect to a vertical plane containing a
  curvature line becomes the rotation with respect to a horizontal
  geodesic of $\htwo$, and viceversa.
\item The symmetry with respect to a horizontal plane containing a
  curvature line becomes the rotation with respect to a vertical
  straight line, and viceversa.
\end{itemize}
\end{lemma}

We will use the above correspondence to study the conjugate surface of
minimal graphs defined on convex polygonal domains of $\htwo$.  The
surface contructed in this way is a minimal graph (and consequently
embedded), as ensured by the following generalized version of Krust's
Theorem.

\begin{theorem}[{\cite[Theorem 14]{HST}}]\label{th:krust}
  Let $X(\Sigma)$ be a (vertical) minimal graph over a convex domain
  $\Omega\subset\htwo$. Then $X^*(\Sigma)$ is a (vertical) minimal
  graph.
\end{theorem}

\section{Saddle Towers in $\htwo\times\R$}
\label{bounded.case}
This section deals with the construction of properly embedded minimal
surfaces in $\htwo\times\R$ invariant by a vertical translation $T$,
which have total curvature $4\pi(1-k)$, genus zero and $2k$ vertical
Scherk-type ends in the quotient by $T$ (Theorem~\ref{th:main}). The
construction is similar to the one of Karcher's Saddle Towers in
$\R^3$.  We also call these new examples {\it Saddle Towers}.

Consider a Jenkins-Serrin domain $\O$ whose edges
$A_1,B_1,A_2,B_2,\ldots,A_k,B_k$ (cyclically ordered) have length
$\ell \in (0,+\infty)$.  Denote by $p_1,\ldots, p_{2k}$ the vertices
of $\Omega$, such that $p_{2i-1},p_{2i}$ are the endpoints of $A_i$
and $p_{2i},p_{2i+1}$ are the endpoints of $B_i$, for $i=1,\ldots,k$
(as usual, we consider the cyclic notation $p_{2k+1}\equiv p_1$), see
Figure~\ref{Fig3}.

\begin{figure}
  \begin{center}
    \epsfysize=8cm \epsffile{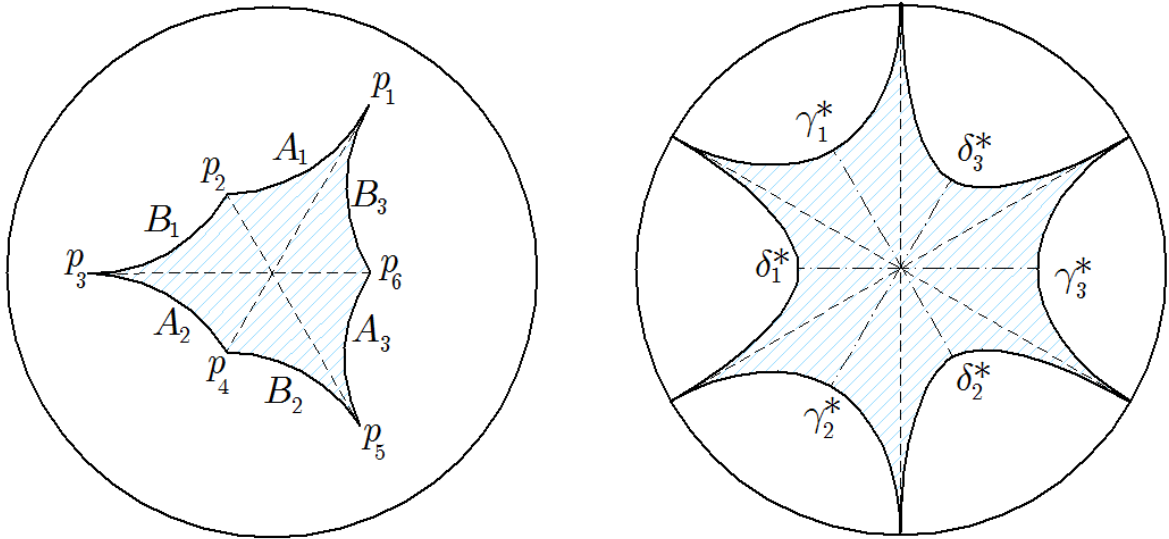}
  \end{center}
  \caption{Left: This is an example of a (symmetric) bounded
    Jenkins-Serrin domain $\Omega$ with six edges of the same
    length. Right: This picture shows the vertical projection over
    $\htwo$ of the conjugate surface (which is a graph) of the graph
    over $\Omega$ with boundary values $+\infty$ over $A_1\cup A_2\cup
    A_3$ and $-\infty$ over $B_1\cup B_2\cup B_3$.}
  \label{Fig3}
\end{figure}

By Theorem~\ref{th:JS}, there exists a solution $u:\O \to \R$ to the
minimal graph equation \eqref{eq.min.surf} on $\O$ satisfying
$u_{|A_i}=+\infty$ and $u_{|B_i}=-\infty$, for any $i=1,\ldots,k$.

The geometry of the graph surface $\Sigma$ of $u$ 
near $\partial\Omega$  is explained
in~\cite{NR}. 
When we
approach a point in $A_i$ (resp. $B_i$) 
within $\O$, the tangent plane
to $\Sigma$ becomes vertical, asymptotic to $A_i\times\R$ (resp.
$B_i\times\R$); i.e. the angle function $\nu$ goes to zero as we
approach $A_i,B_i$.  
Moreover $\Sigma$ is bounded by $2k$ vertical straight
lines passing through the vertices of $\O$.
Since $\Sigma$ is asymptotic to $A_i\times\R$ (resp. $B_i\times\R$)
over $A_i$ (resp. $B_i$), the intrinsic distance on $\Sigma$ from
$\{p_{2i-1}\}\times\R$ 
to $\{p_{2i}\}\times\R$ 
(resp. from $\{p_{2i}\}\times\R$ 
to $\{p_{2i+1}\}\times\R$) 
is $\ell$, which is never attained ($\ell$ is the asymptotic intrinsic
distance at infinity).

\begin{proposition}
  \label{parallel.planes}
  The conjugate surface $\Sigma^*$ of $\Sigma$ is a (vertical) minimal
  graph, whose boundary is of the form
  $\partial\Sigma^*=\g_1^*\cup\delta_1^*\cup\ldots\cup\g_k^*\cup\delta_k^*$,
  where
  \[
  \g_1^*,\ldots,\g_k^* \subset \{t=0\}\qquad \mbox{and}\qquad
  \delta_1^*,\ldots,\delta_k^* \subset \{t=\ell\}
  \]
  are  curvature lines of symmetry. 
  Let us call $\O^*$, $\widetilde\delta_i^*$ the respective vertical
  projection of $\Sigma^*$, $\delta_i^*$ over $\{t=0\}$. Then:
  \begin{enumerate}
  \item None of the curves $\g_i^*,\widetilde\delta_i^*$ is convex
    (with respect to $\O^*$) at any point.
  \item $\g_i^*$ and $\widetilde\delta_i^*$
    (resp. $\widetilde\delta_i^*$ and $\g_{i+1}^*$) are asymptotic 
    at their common endpoint at $\partial_\infty\htwo$.
  \item
    $\partial\O^*=\g_1^*\cup\widetilde\delta_1^*\cup\ldots\cup\g_k^*\cup\widetilde\delta_k^*$
    (cyclically ordered).
  \item $\Sigma^*-\partial\Sigma^*\subset\{0<t<\ell\}$.
  \end{enumerate}
\end{proposition}

\begin{proof}
  Since $\O$ is convex, Theorem \ref{th:krust} says that the conjugate
  surface $\Sigma^*$ of $\Sigma$ is a minimal graph over a domain
  $\O^*\subset\htwo$.  By Lemma~\ref{lem:Schwarz}, we know that the
  conjugation transforms vertical straight lines into horizontal
  curvature lines of symmetry.  Then $\partial \Sigma^*$ consists of
  $2k$ horizontal symmetry curves
  $\g^*_1,\delta^*_1,\ldots,\g^*_k,\delta^*_k$. Assume those boundary
  curves are ordered so that two consecutive ones correspond by
  conjugation to vertical straight lines in $\partial\Sigma$ through
  consecutive vertices of $\O$.
  For every $i=1,\ldots,k$, let $\g_i,\delta_i\subset\partial\Sigma$
  be the straight lines which correspond by conjugation to
  $\g^*_i,\delta^*_i\subset\partial\Sigma^*$; and
  $\widetilde\g_i^*,\widetilde\delta_i^*$ be the vertical projection
  of $\g_i^*,\delta_i^*$ over $\{t=0\}\equiv\htwo$, respectively.

  Consider the surface $M$ obtained by extending $\Sigma^*$ by
  symmetry with respect to the horizontal plane containing $\g^*_1$.
  If $\widetilde\g^*_1$ is convex (with respect to $\O^*$) at some
  point, then we will obtain by the maximum principle that $M$ is
  contained in a vertical plane, a contradiction with the fact that
  $\Sigma^*$ is a graph. Similarly, we deduce that none of vertical
  projections of the curves in $\partial\Sigma^*$ has a convexity
  point. This proves {\it (1)}.

  The asymptotic intrinsic distance at infinity between
  $\g_i,\delta_i$ (resp.  $\delta_i,\g_{i+1}$) is $\ell$ since the
  surface $\Sigma$ is asymptotically vertical. By
  Theorem~\ref{th:conjugate}, $\Sigma,\Sigma^*$ are isometric and have
  the same angle function.  Thus the asymptotic intrinsic distance
  between $\g_i^*,\delta_i^*$ (resp. $\delta_i^*,\g_{i+1}^*$) is
  $\ell$, and the unit normal vector field $N^*$ to $\Sigma^*$ is
  asymptotically horizontal between $\g_i^*$ and $\delta_i^*$ (resp.
  between $\delta_i^*$ and $\g_{i+1}^*$). In particular,
  $\widetilde\delta_i^*$ shares an endpoint with $\widetilde\g_i^*$,
  where they arrive tangentially, and the other with
  $\widetilde\g_{i+1}^*$, where they are also tangent. Observe that
  this proves {\it (3)}.

  To finish {\it (2)}, it remains to prove that the endpoints of each
  $\widetilde\delta_i^*$ are at $\partial_\infty\htwo$.  Fix a point
  $p^*\in\delta_i^*$, and let $\widetilde p^*\in\widetilde\delta_i^*$
  be its vertical projection.  The point $p^*$ corresponds by
  conjugation to a point $p\in\delta_i$, which divides $\delta_i$ in
  two curves of infinite length. Since $\Sigma$ and $\Sigma^*$ are
  isometric, then each component of $\widetilde\delta_i^*-\{\widetilde
  p^*\}$ has infinite length as well, and finishes at a common
  endpoint with $\widetilde\g_{i}^*$ or $\widetilde\g_{i+1}^*$. Since
  all $\widetilde\g_{i}^*,\widetilde\delta_i^*,\widetilde\g_{i+1}^*$
  are not convex (with respect to $\O^*$) at any point, we deduce that
  the endpoints of $\widetilde\delta_i^*$ must be at
  $\partial_\infty\htwo$.

  Recall that the asymptotic intrinsic distance between
  $\g_i^*,\delta_i^*$ (resp. $\delta_i^*,\g_{i+1}^*$) is $\ell$, for
  any $i=1,\ldots,k$.  We can assume that $\g_1^* \subset \{t=0\}$ and
  $\delta_1^*\subset \{t=\ell\}$. We know that either $\g_2^* \subset
  \{t=0\}$ or $\g_2^* \subset \{t=2\ell\}$.  Let us prove that the
  second case is impossible.  We call $q_1^*\in\partial_\infty\htwo$
  the common endpoint of $\widetilde\g_1^*,\widetilde\delta_1^*$.
  Since $\Sigma^*$ is asymptotic to $\{q_1^*\}\times(0,\ell)$ when we
  approach $q_1^*$ within $\O^*$, then $\Sigma^*$ is locally below
  $\delta_1^*$ near the asymptotic point $(q_1^*,\ell)$ at infinity.
  Since $N^*$ is horizontal along $\delta_1^*$, then $\Sigma^*$ is
  locally below $\{t=\ell\}$ in a small neighborhood of $\delta_1^*$.
  In particular, $\Sigma^*$ is locally below $\delta_1^*$ near
  $(q_2^*,\ell)$, where $q_2^*$ is the common endpoint of
  $\widetilde\delta_1^*,\widetilde\g_2^*$. Thus $\Sigma^*$ cannot be
  asymptotic to $\{q_2^*\}\times(\ell,2\ell)$, and then $\g_2^*
  \subset \{t=0\}$. Arguing similarly we prove
  \[
  \g_1^*,\ldots,\g_k^* \subset \{t=0\}\qquad \mbox{and}\qquad
  \delta_1^*,\ldots,\delta_k^* \subset \{t=\ell\}.
  \]

  Finally we obtain {\it (4)} by the maximum principle.
\end{proof}

Since $\Sigma^*$ is a graph, it is in particular embedded. By
reflecting $\Sigma^*$ about the horizontal plane $\{t=\ell\}$ we get a
surface $M$ whose boundary consists of horizontal curvature lines of
symmetry at heights $0$ and $2\ell$, which differ by the translation
by $T=(0,0,2\ell)$. Moreover, $M$ is embedded, as
$\Sigma^*-\partial\Sigma^*\subset\{0<t<\ell\}$.

Extending $\Sigma^*$ by symmetry with respect to the horizontal planes
at heights multiple of $\ell$, we obtain an embedded singly periodic
minimal surface $\mathcal{M}$ with period $T=(0,0,2\ell)$.
Furthermore, $\mathcal{M}$ is proper, by item {\it (2)} in
Proposition~\ref{parallel.planes}.  It is easy to see that the
quotient of $\mathcal{M}$ by $T$ has genus $0$ and $2k$ ends
asymptotic to flat vertical annuli (named vertical Scherk-type ends).
$\mathcal{M}$ is called a {\it Saddle Tower}.

Moreover, Nelli and Rosenberg~\cite{NR} proved that $\Sigma$ has total
curvature $2\pi(1-k)$. Thus the same holds for $\Sigma^*$, and the
fundamental domain $M$ of $\mathcal{M}$ has total curvature
$4\pi(1-k)$.

To finish Theorem~\ref{th:main}, it remains to prove that, given
$k\geqslant 2$ and $\ell\in(0,+\infty)$, there exists $(2k-3)$
possible Jenkins-Serrin domains $\O$ with $2k$ edges
$A_1,B_1,A_2,B_2,\ldots,A_k,B_k$ of length $\ell$, after identifying
them by isometries of $\htwo$.  Up to an isometry of $\htwo$ we can
assume that $A_1$ is fixed; i.e. the vertices $p_1,p_2$ are fixed.
Observe that once we have chosen the vertices $p_3,\ldots,p_{2k-1}$,
then the vertex $p_{2k}$ is determined by $p_{2k-1}$ and $p_1$, as
$\O$ is a Jenkins-Serrin domain.  Each vertex $p_i$,
$i=3,\ldots,2k-1$, is at distance $\ell$ from $p_{i-1}$, hence $p_i$
is determined by the interior angle $\theta_{i-1}$ at $p_{i-1}$ (i.e.
the interior angle at $p_{i-1}$ between the edges in $\partial\O$
which have $p_{i-1}$ as a common endpoint). Since $\O$ is convex,
$0\leqslant \theta_{i-1}\leqslant\pi$. Additional constraints for
$\theta_{i-1}$ come from the facts that $\partial\O$ is closed, and
that $\O$ is a Jenkins-Serrin domain.  We have obtained that the space
of Jenkins-Serrin domains $\O$ with $2k$ edges of length $\ell$, one
of them fixed, has $2k-3$ freedom parameters $\t_2,\ldots,\t_{2k-2}$.
This proves Theorem~\ref{th:main}.

\begin{remark}[Symmetric case]\label{rem:symmetry}
  {\rm \begin{enumerate}
    \item[]
    \item In the case the vertices $p_{2i-1}$ of $\O$ are at distance
      $\lambda$ from a point of $\htwo$, say the origin ${\bf 0}$, and
      the vertices $p_{2i}$ of $\O$ are at distance $\mu$ from ${\bf
        0}$ (see Figure~\ref{Fig3}, left), then the graph $\Sigma$
      over $\O$ can be obtained (up to a vertical translation) by
      reflection from the minimal graph $\Sigma_T$ over a triangle $T$
      with vertices $p_1,p_2,{\bf 0}$ with boundary values $+\infty$
      along $A_1$ and $0$ along $\partial T-A_1$. Then $\Sigma$
      contains $k$ geodesic arcs at height $0$ meeting at ${\bf
        0}\in\Sigma$ equiangularly (as usual, we are identifying
      $\htwo\equiv\htwo\times\{0\}$):
   \begin{itemize}
   \item $k$ geodesic arcs of length $\lambda+\mu$, if $k$ is odd;
   \item $k/2$ geodesic arcs of length $2\lambda$ and $k/2$ geodesic
     arcs of length $2\mu$, when $k$ is even.
   \end{itemize}
   Those horizontal geodesics give us by conjugation $k$ vertical
   curvature lines of symmetry in $\Sigma^*$ (of the same length as in
   $\Sigma$) meeting with angle $\pi/k$.  Then $\Sigma^*$ can be
   obtained from $\Sigma_T^*$ by symmetries. 

 \item By uniqueness of the Jenkins-Serrin graphs, when $\lambda=\mu$
   we have that $\Sigma_T$ is symmetric with respect to the vertical
   plane which bisects $T$ at its vertex ${\bf 0}$.  That symmetry
   says that $\Sigma_T^*$ contains a horizontal straight line, and
   then we can obtain $\Sigma^*$ from half a $\Sigma_T^*$ bounded by a
   horizontal curvature line of symmetry $\gamma$ at height $0$, a
   vertical curvature line of symmetry $\alpha$ and a horizontal
   geodesic curve $L$ at height $\ell/2$; $\a,L$ meeting at an angle
   $\pi/(2k)$. All those symmetries allow us to obtaining the
   corresponding Saddle Tower more easily than in the general
   case. These symmetric examples are the ones Lee and Pyo are
   constructing~\cite{LP}.
 \end{enumerate}}
\end{remark}

\section{Properly embedded minimal surfaces of genus
  zero in $\htwo\times\R$}
\label{semi.ideal}

In this section we obtain as a limit of Saddle Towers with $2k$
vertical Scherk-type ends and period vector $(0,0,2\ell)$, with
$\ell\to+\infty$, a properly embedded minimal surface in
$\htwo\times\R$ with total curvature $4\pi(1-k)$, genus zero and $k$
ends asymptotic to vertical geodesic planes (Theorem~\ref{th:limits}).
It will be the conjugate surface of a Jenkins-Serrin graph over a
semi-ideal polygonal domain.

Consider a semi-ideal Jenkins-Serrin domain $\O$ with $2k$ vertices
$p_1,\ldots,p_{2k}$ cyclically ordered so that the vertices $p_{2i-1}$
are in the interior of $\htwo$, and the vertices $p_{2i}$ are at
$\partial_\infty{\mathbb H}^2$, for $i=1,\ldots,k$. As in the previous
section, call $A_i$ the edge of $\O$ whose endpoints are
$p_{2i-1},p_{2i}$, and $B_i$ the edge of $\O$ whose endpoints are
$p_{2i},p_{2i+1}$.  We also require that $\O$ satisfies the following
additional condition:

\begin{quote}
  ($\star$)\quad For each $p_{2i}\in \partial_\infty\htwo$, there
  exists a sufficiently small horocycle $H_{2i}$ such that it only
  intersects $\partial\O$ along $A_i,B_i$, and
  \[
  \mbox{dist}_\htwo(p_{2i-1},H_{2i})=\mbox{dist}_\htwo(p_{2i+1},H_{2i}).
  \]
\end{quote}
Observe we can choose the horocycles $H_{2i}$ so that, for any
$i=1,\ldots,k$,
\[
\mbox{dist}_\htwo(p_{2i-1},H_{2i})=\mbox{dist}_\htwo(p_{2i+1},H_{2i})=\ell_0
,
\]
independently of~$i$. Also we can choose them small enough so that
\begin{equation}
\label{condition.js}
\mbox{dist}_\htwo(p_{2i-1},p_{2i+1})<2\ell_0,\quad \mbox{for any } i=1,\ldots,k .
\end{equation}

Consider the nested sequence of horocycles $H_{2i}(n)$ at $p_{2i}$,
$n\geqslant 0$, converging to $p_{2i}$ as $n\to+\infty$, such that
$H_{2i}(0)=H_{2i}$ and
$\mbox{dist}_\htwo\left(H_{2i}(n+1),H_{2i}(n)\right)=1$.  We set
\[
\ell_n=\ell_0+n .
\]
We are going to obtain $\O$ as limit of Jenkins-Serrin domains $\O_n$
as $n\to+\infty$, each $\O_n$ with $2k$ edges of length $\ell_n$.

Firstly, we remark the following fact. Condition ($\star$) ensures the
existence of a horocycle $C_{2i}$ at $p_{2i}$ passing through
$p_{2i-1},p_{2i+1}$. Call $D_{2i}$ the component of $\htwo-C_{2i}$
whose only point of $\partial_\infty\htwo$ at its infinite boundary
is~$p_{2i}$ (i.e. $D_{2i}$ is the domain ``inside'' the horocycle
$C_{2i}$), and $\overline{D_{2i}}=D_{2i}\cup C_{2i}$. We get the
following lemma, since $\O$ is a Jenkins-Serrin domain.

\begin{lemma}\label{lem:horocycles}
  Every vertex $p_{2j-1}$ of $\Omega$, for $j\not\in\{i,i+1\}$, is
  contained in $\htwo-\overline{D_{2i}}$.
\end{lemma}
\begin{proof}
  Suppose there exists some $p_{2j-1}\in \overline{D_{2i}}$, with
  $j\not\in\{i,i+1\}$.  Then, for every $n$, we have
  \[
  \mbox{dist}_\htwo(p_{2j-1},H_{2i}(n))\leqslant\ell_n=
  \mbox{dist}_\htwo(p_{2i-1},H_{2i}(n))
  =\mbox{dist}_\htwo(p_{2i+1},H_{2i}(n)).
  \]
  Let $\gamma$ be the geodesic from $p_{2j-1}$ to $p_{2i}$, and
  $\mathcal{P}$ be the component of $\O-\gamma$ containing $A_i$ on
  its boundary. Clearly, $\mathcal{P}$ is a polygonal domain inscribed
  in $\O$. It holds $\beta(\mathcal{P})=\a(\mathcal{P})-\ell_n$ for
  this choice of horocycles, and then
  \[
  |\Gamma(\mathcal{P})|=\mbox{dist}_\htwo(p_{2j-1},H_{2i}(n))
  +\a(\mathcal{P})+\beta(\mathcal{P})\leqslant 2\a(\mathcal{P}).
  \]
  And this holds for every $n$, a contradiction as $\O$ is a
  Jenkins-Serrin domain.
\end{proof}

Now let us construct the Jenkins-Serrin domains $\O_n$. 
All the vertices $p_{2i-1}\in\htwo$ of $\O$ will be vertices of each
$\O_n$ as well. Let us obtain the vertices $p_{2i}(n)$ of $\O_n$ such
that, for each $i=1,\ldots,k$:
\begin{enumerate}
\item[(a)]
  $\mbox{dist}_\htwo(p_{2i}(n),p_{2i-1})=\mbox{dist}_\htwo(p_{2i}(n),p_{2i+1})=\ell_n$;
\item[(b)] $p_{2i}(n)\in\O$ and $p_{2i}(n)\to p_{2i}$ as $n\to+\infty$.
\end{enumerate}
By \eqref{condition.js},
$\mbox{dist}_\htwo(p_{2i-1},p_{2i+1})<2\ell_0<2\ell_n$. This
guarantees that the circles of radius $\ell_n$ centered at
$p_{2i-1},p_{2i+1}$ intersect at exactly two points, each one lying in
a different component of $\htwo-\g_{2i}$, where $\g_{2i}$ is the
complete geodesic passing through $p_{2i-1},p_{2i+1}$, see
Figure~\ref{Fig4}. We define $p_{2i}(n)$ as the intersection point of
those circles which is contained in the component of $\htwo-\g_{2i}$
having $p_{2i}$ at its boundary at infinity.  The point $p_{2i}(n)$
lies in the region of $\O$ bounded by $A_i,H_{2i}(n),B_i$ and
$\g_{2i}$.  By construction, $p_{2i}(n)$ verifies conditions (a) and
(b) above, and the constructed Jenkins-Serrin domain $\O_n$ has $2k$
edges of length $\ell_n$ and converges to $\O$ as $n\to+\infty$.  Call
$A_i(n)$ the edge of $\O_n$ whose endpoints are $p_{2i-1},p_{2i}(n)$,
and $B_i(n)$ the edge of $\O_n$ whose endpoints are
$p_{2i}(n),p_{2i+1}$.

\begin{figure}
  \begin{center}
    \epsfysize=8cm \epsffile{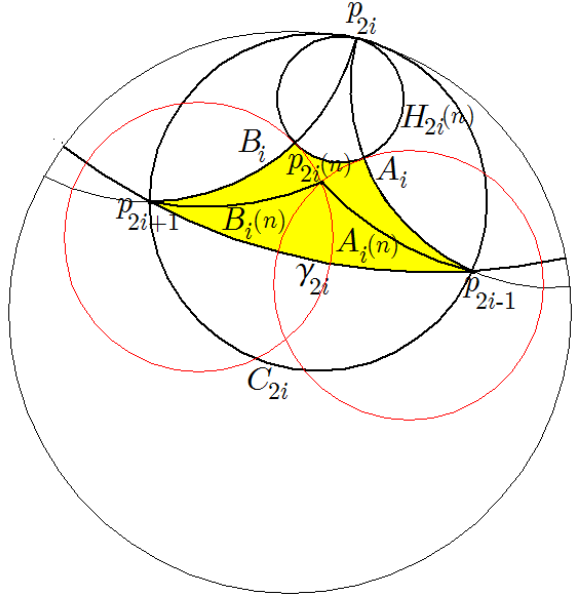}
  \end{center}
  \caption{Construction of the vertex $p_{2i}(n)$ of $\Omega_n$ as the
    intersection point in the shadowed region of the circles of radius
    $\ell_n$ centered at $p_{2i-1},p_{2i+1}$.}
  \label{Fig4}
\end{figure}

\begin{lemma}\label{lem:JS}
  For $n$ big enough, $\O_n$ is a Jenkins-Serrin domain.
\end{lemma}
\begin{proof}
  By construction, $\a(\O_n)=\b(\O_n)$.  Suppose there exists an
  inscribed polygonal domain $\mathcal{P}$ in $\O_n$,
  $\mathcal{P}\neq\O_n$, such that $|\partial\mathcal{P}| \leqslant
  2\a(\mathcal{P})$ (the case $|\partial\mathcal{P}| \leqslant
  2\b(\mathcal{P})$ follows similarly).  Since $\mathcal{P}\neq\O_n$,
  there is at least an interior geodesic $\g_1$ in
  $\partial{\mathcal{P}}$ (i.e.
  $\g_1\subset\partial{\mathcal{P}}\cap\O_n$). We can assume that
  there are no two consecutive interior geodesics $\g_1,\g_2$: we
  would replace $\mathcal{P}$ by another inscribed polygonal domain
  satisfying the same properties by replacing the geodesics
  $\g_1,\g_2$ by the geodesic $\g_3$ such that $\g_1\cup\g_2\cup\g_3$
  bound a geodesic triangle contained in $\O_n$. In a similar way, we
  can assume that
  \[
  \partial{\mathcal{P}}= A_{i_1}(n)\cup\g_1\cup\ldots\cup
  A_{i_j}(n)\cup\g_j\cup A_{i_{j+1}}(n)\cup\ldots\cup
  A_{i_s}(n)\cup\g_s ,
  \]
  where each $\g_j$ is either an interior geodesic or a $B_i(n)$ edge,
  and at least $\g_1\subset\O_n$.  In particular, each $\g_j$ joins an
  even vertex $q_{2j}(n)=p_{2i_j}(n)$ to an odd vertex $q_{2j+1}=p_{2
    i_{j+1}-1}$.  Remark that when $\g_j$ is a $B_i(n)$ edge, then
  $\g_j=B_{i_j}(n)$ and $i_{j+1}=i_j+1$.

  As $\sum_{j=1}^s |\g_j|= |\partial\mathcal{P}|-\a(\mathcal{P})
  \leqslant \a(\mathcal{P})=s\ell_n$, there must be some interior
  geodesic $\g_j\subset\partial{\mathcal{P}}$ whose length is smaller
  than or equal to $\ell_n$. Take the hyperbolic circle $S(n)$ of
  center $q_{2j}(n)$ and radius $\ell_n$, and let $D(n)$ be the
  hyperbolic disk bounded by $S(n)$, see Figure~\ref{Fig5}. Then the
  vertex $q_{2j+1}$ lies in $\overline{D(n)}=D(n)\cup S(n)$. Let us
  prove that this is not possible when $n$ is large.

  \begin{figure}
    \begin{center}
      \epsfysize=8cm \epsffile{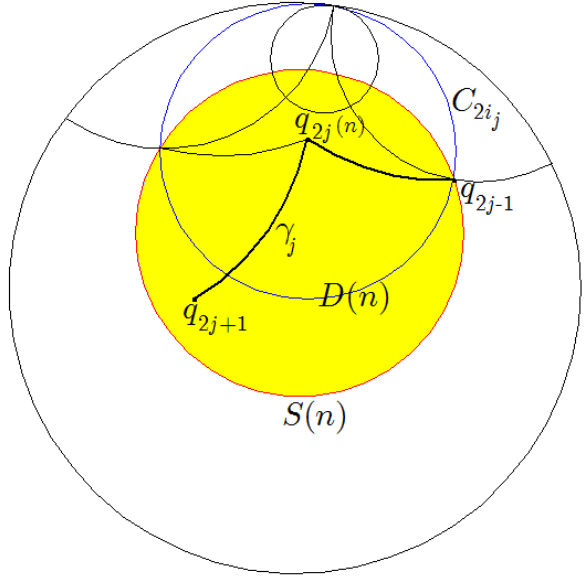}
    \end{center}
    \caption{The circle $S(n)$ of radius $\ell_n$ centered at
      $q_{2j}(n)$ converges to the horocycle $C_{2i_j}$ at $p_{2i_j}$
      as $n\to+\infty$.}
    \label{Fig5}
  \end{figure}

  The circles $S(n)$ converge to the horocycle $C_{2i_j}$ as
  $n\to+\infty$. And by Lemma~\ref{lem:horocycles}, $q_{2j+1}$ cannot
  be contained in the closed horodisk $\overline{D_{2i_j}}$ bounded by
  $C_{2i_j}$. Then $q_{2j+1}\in\htwo-\overline{D(n)}$ for $n$ large
  enough.
\end{proof}

Observe that the inscribed polygonal domain $\mathcal{P}_0$ whose
vertices are the vertices $p_{2i-1}$ of $\O$, is contained in all the
domains $\O_n$. Fix a point $p_0\in\mathcal{P}_0$.

By Theorem~\ref{th:JS}, there exists a solution $u$ (resp. $u_n$, for
any $n$) to the minimal graph equation defined over $\O$ (resp.
$\O_n$) with boundary values $+\infty$ over $A_i$ (resp. $A_i(n)$) and
$-\infty$ over $B_i$ (resp. $B_i(n)$). Denote by $\Sigma$ (resp.
$\Sigma_n$) the graph surface of $u$ (resp. $u_n$). Up to a vertical
translation we can assume $u(p_0)=u_n(p_0)=0$.  (Observe that we could
have exchanged the edges $A_i,B_i$, but the graph we would have
obtained would be $\Sigma$ up to a symmetry about $\htwo\times\{0\}$).

The domains $\O_n$ converge to $\O$. Since $\O$ is a Jenkins-Serrin
domain, there cannot exist divergence lines associated to the sequence
of graphs $u_n$ (see~\cite{MRR} for the definition of divergence lines
and for similar arguments). Since $u_n(p_0)=0$ for any $n\in\N$, a
subsequence of the $u_n$ converges uniformly on compact sets of $\O$
to a solution $u_\infty$ of the minimal graph equation. We can deduce
$u_\infty=u$.  Hence the minimal graphs $\Sigma_n$ converge to
$\Sigma$, after taking a subsequence.

Observe that the vertical straight lines
$\Gamma_i=\{p_{2i-1}\}\times\R$ are contained in the boundary of all
the $\Sigma_n$ and also of $\Sigma$. In particular, none of the
distances $\mbox{dist}_{\Sigma_n}(\Gamma_i,\Gamma_j)$ can diverge, for
any $i,j\in\{1,\ldots,k\}$. After passing to a subsequence, we can
assume that there exists a constant $C>0$ such that, for
  any $i,j\in\{1,\ldots,k\}$,
\[
\mbox{dist}_{\Sigma_n}(\Gamma_i,\Gamma_j)\leqslant C, \quad
\mbox{ for any } n\in\N,
\]
and $\mbox{dist}_{\Sigma}(\Gamma_i,\Gamma_j)\leqslant C$.

Denote by $\Sigma^*$ (resp. $\Sigma_n^*$) the conjugate surface of
$\Sigma$ (resp. $\Sigma_n$).  By Theorem~\ref{th:krust}, $\Sigma^*$ is
a minimal graph, as $\Sigma$ is a minimal graph over $\O$, which is
convex.  Moreover, $\partial\Sigma=\Gamma_1\cup\ldots\cup\Gamma_k$, so
the boundary of $\Sigma^*$ is composed of $k$ horizontal curvature
lines of symmetry $\Gamma_i^*$, by Lemma~\ref{lem:Schwarz}.  Since
$\Sigma,\Sigma^*$ are isometric, then
$\mbox{dist}_{\Sigma^*}(\Gamma_i^*,\Gamma_j^*)\leqslant C$ for every
$i,j\in\{1,\ldots,k\}$.  We want to prove that all the curves
$\Gamma_i^*$ lie in the same horizontal plane, say $\{t=0\}$, and
$\Sigma^*$ is contained in one of the half-spaces determined by
$\{t=0\}$.

For any $n$, $\Sigma_n$ is a graph over the convex domain
$\O_n$ and the boundary of $\Sigma_n$ equals
$\Gamma_1\cup\eta_1(n)\cup\ldots\Gamma_k\cup\eta_k(n)$, where each
$\Gamma_i$ is defined as above and $\eta_i(n)=\{p_{2i}(n)\}\times\R$,
for any $i=1,\ldots,k$.  By Proposition~\ref{parallel.planes},
$\Sigma_n^*$ is a graph over a domain $\O_n^*$ and 
\[
\partial\Sigma_n^*=\Gamma_1^*(n)\cup\eta_1^*(n)\cup\ldots\cup
\Gamma_k^*(n)\cup\eta_k^*(n) ,
\]
where $\Gamma_1^*(n),\ldots,\Gamma_k^*(n)$
(resp. $\eta_1^*(n),\ldots,\eta_k^*(n)$) are horizontal curvature
lines of symmetry contained in the same horizontal plane, and both
planes are at distance $\ell_n$ from each other.

Call $\widetilde\Gamma_i^*(n)$ (resp. $\widetilde\eta_i^*(n)$) the
vertical projection of $\Gamma_i^*(n)$ (resp. $\eta_i^*(n)$) over
$\{t=0\}$. Then
$\partial\O_n^*=\widetilde\Gamma_1^*(n)\cup\widetilde\eta_1^*(n)\cup\ldots\cup
\widetilde\Gamma_k^*(n)\cup\widetilde\eta_k^*(n)$, and two consecutive
curves in $\partial\O_n^*$ are asymptotic at $\partial_\infty\htwo$.
We know that, up to a vertical translation to have a fixed point, the
graphs $\Sigma_n^*$ converge to the graph $\Sigma^*$. It could be that
the boundary values of the graphs over the boundary curves
$\widetilde\Gamma_i^*(n)$ would diverge to $-\infty$; but this is not
possible as
$\mbox{dist}_{\Sigma_n^*}(\Gamma_i^*(n),\Gamma_j^*(n))\leqslant C$,
for any $i,j\in\{1,\ldots,k\}$ and any $n\in\N$.  Therefore, up to a
vertical translation we can assume that
$\Gamma_1^*(n),\ldots,\Gamma_k^*(n)\subset\{t=0\}$; the curves
$\widetilde\Gamma_i^*(n)=\Gamma_i^*(n)$ converge to the curves
$\Gamma_i^*$, and the graphs $\Sigma_n^*$ converge to~$\Sigma^*$. In
particular, $\partial\Sigma^*\subset\{t=0\}$.

Suppose $\eta_1^*(n),\ldots,\eta_k^*(n)\subset\{t=\ell_n\}$ (if they
are contained in $\{t=-\ell_n\}$ we argue similarly). The height of
each curve $\eta_i^*(n)$ diverge to $+\infty$, hence its projection
$\widetilde\eta_i^*(n)$ converge to the geodesic
$\widetilde\eta_i^*\in\htwo$ joining the corresponding endpoints of
$\Gamma_i^*,\Gamma_{i+1}^*$ at $\partial_\infty\htwo$.  Moreover,
since the graph $\Sigma^*(n)$ is contained in $\{t\geqslant 0\}$ for
any $n$, then the same holds for $\Sigma^*$. If we reflect
$\Sigma^*(n)$ with respect to $\{t=0\}$, we get a properly embedded
minimal surface $M$ of genus zero and $k$ planar ends in
$\htwo\times\R$ (the ends of $M$ is asymptotic to the vertical
geodesic planes $\widetilde\eta_i^*\times\R$).

Collin and Rosenberg~\cite{CR} proved that $\Sigma$ (and so
$\Sigma^*$) has total curvature $2\pi(1-k)$. Hence $M$ has total
curvature $4\pi(1-k)$.

To complete the proof of Theorem~\ref{th:limits}, it remains to show
that, given $k\geqslant 2,$ there exist $(2k-3)$ possible semi-ideal
Jenkins-Serrin domains $\O$ satisfying condition ($\star$),
after identifying them by isometries of $\htwo$.
Firstly we observe that, for $i=1,\ldots,k,$ the vertex $p_{2i} \in
\partial_\infty \htwo$ is determined once we have chosen $p_{2i-1}$
and $p_{2i+1}$, since $\O$ satisfies condition~($\star$). Thus we have
to compute the parameters which determine the vertices having odd
subindex.  We can assume that $p_1$ is fixed as well as the direction
of the geodesic arc $\a_1$ from $p_1$ to $p_3$. So the vertex $p_3$ is
determined by the length of $\a_1$. That gives the first parameter.
For $i=2,\ldots,k-1$, the vertex $p_{2i+1}$ is determined by both the
direction and the length of the geodesic arc $\a_i$ joining
$p_{2i-1},p_{2i+1}$ (the direction of $\a_i$ is given by the interior
angle at $p_{2i-1}$ between $\a_{i-1}$ and $\a_i$). So we have two
additional parameters for the remaining $k-2$ vertices of $\O$, and
the total number of freedom parameters equals $2k-3.$ This finishes
Theorem~\ref{th:limits}.

\begin{remark}[Symmetric case]\label{rem:symmetrySI}
  {\rm In the case the vertices $p_{2i-1}$ of $\O$ are at distance
    $\lambda$ from the origin ${\bf 0}$ of $\htwo$, then we can get
    $\Sigma^*$ as limit of symmetric surfaces as at item (1) of
    Remark~\ref{rem:symmetry} (when $\mu\to+\infty$). In particular,
    $\Sigma^*$ contains $k$ vertical curvature lines of symmetry (of
    infinite length when $k$ is odd; or $k/2$ of length $2\lambda$ and
    $k/2$ of infinite length when $k$ is even).  These are the
    examples constructed by Pyo in~\cite{P}.}
\end{remark}

\section{Jenkins-Serrin graphs over ideal polygonal domains}
\label{ideal}

In this section we prove Theorem~\ref{th:limitsJS}.  Let $\O$ be an
ideal Jenkins-Serrin domain with $2k$ vertices
$p_1,\ldots,p_{2k}\in\partial_\infty\htwo$ cyclically ordered.  We are
going to obtain $\O$ as limit of semi-ideal Jenkins-Serrin domains
$\O_n$ contained in $\O$. Using the previous section and a diagonal
argument, that fact ensures we can also obtain $\O$ as limit of
bounded Jenkins-Serrin domains $\widetilde\O_n$, each $\widetilde\O_n$
with $2k$ edges of the same length.

As in the previous sections, call $A_i$ the edge of $\O$ whose
endpoints are $p_{2i-1},p_{2i}$, and $B_i$ the edge of $\O$ whose
endpoints are $p_{2i},p_{2i+1}$.  For each $i=1,\ldots,2k$, we can
consider a horocycle $H_i$ at $p_i$ such that it only intersects
$\partial\O$ along the edges of $\O$ finishing at $p_i$, and
  \[
  \mbox{dist}_\htwo(H_i,H_{i-1})=\ell_0,
\]
for some $\ell_0>0$, as $\O$ satisfies the Jenkins-Serrin conditions.
Consider the nested sequence of horocycles $H_{i}(n)$ at $p_{i}$, with
$n\geqslant 0$, converging to $p_{i}$ as $n\to+\infty$, such that
$H_{2i}(0)=H_{2i}$ and
\[
\mbox{dist}_\htwo\left(H_{i}(n+1),H_{i}(n)\right)=1 .
\]
Define
\[
\delta_n= \max_{i=1,\ldots,k} \{\mbox{dist}_\htwo(B_{i-1}\cap
H_{2i-1}(n), A_i\cap H_{2i-1}(n)), \mbox{dist}_\htwo(A_i\cap
H_{2i}(n),B_i\cap H_{2i}(n))\} ,
\]
which converges to zero as $n$ goes to $+\infty$, and
\[
\ell_n=\ell_0+2n+2\delta_n,
\]
which diverges to $+\infty$ as $n\to+\infty$.

The even vertices $p_{2i}$ of $\O$ will be vertices of each $\O_n$ as
well. Let us define the vertices $p_{2i-1}(n)$ of $\O_n$ converging to
$p_{2i-1}$. Consider the horocycle $C_{2i}(n)$ at $p_{2i}$ such that
\[
\mbox{dist}_\htwo\left(C_{2i}(n),H_{2i}(n)\right)=\ell_n.
\]
Since $\mbox{dist}_\htwo\left(H_{2i-2}(n),H_{2i}(n)\right) < 2\ell_n$,
then the horocycles $C_{2i-2}(n),C_{2i}(n)$ intersect at exactly two
points, each one lying in a different component of $\htwo-\g_{2i-1}$,
where $\g_{2i-1}$ is the complete geodesic from $p_{2i-2}$ to
$p_{2i}$, see Figure~\ref{Fig6}.  We define $p_{2i-1}(n)$ as the
intersection point of those horocycles which is contained in the
component of $\htwo-\g_{2i-1}$ having $p_{2i-1}$ at its infinite
boundary. Let $\Upsilon_{2i-1}(n)$ be the horocycle at $p_{2i-1}$
passing through the points $B_{i-1}\cap C_{2i-2}(n)$ and $A_i\cap
C_{2i}(n)$.  The point $p_{2i-1}(n)$ lies in the region of $\O$
bounded by $B_{i-1},\Upsilon_{2i-1}(n),A_i$ and $\g_{2i}$, and
converges to $p_{2i-1}$ when $n$ goes to $+\infty$.  We define $\O_n$
as the semi-ideal polygonal domain whose vertices are $p_1(n),p_2,
\ldots, p_{2k-1}(n),p_{2i}$.  Call $A_i(n)$ the edge of $\O_n$ whose
endpoints are $p_{2i-1}(n),p_{2i}$, and $B_i(n)$ the edge of $\O_n$
whose endpoints are $p_{2i},p_{2i+1}(n)$.
\begin{figure}
  \begin{center}
    \epsfysize=8cm \epsffile{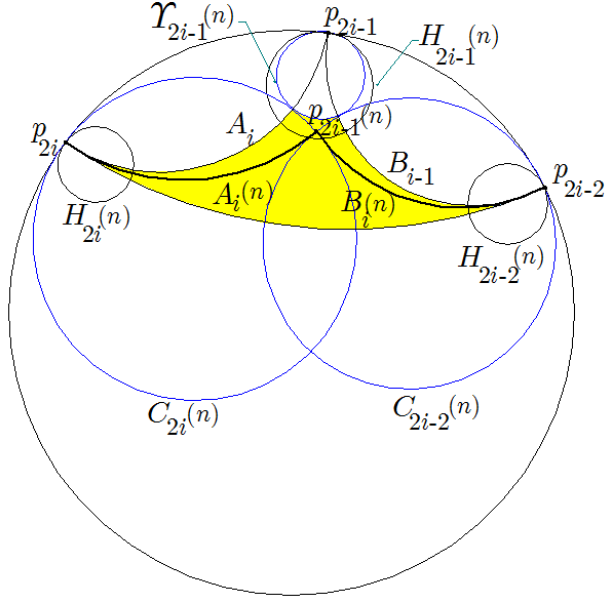}
  \end{center}
  \caption{$\mbox{dist}_\htwo\left(H_{2i-2}(n),H_{2i-1}(n)\right)=
    \mbox{dist}_\htwo\left(H_{2i-1}(n),H_{2i}(n)\right)=\ell_0+2n$ and
    $\mbox{dist}_\htwo\left(C_{2i-2}(n),H_{2i-2}(n)\right)=
    \mbox{dist}_\htwo\left(C_{2i}(n),H_{2i}(n)\right)=\ell_n$.  We get
    the vertex $p_{2i-1}(n)$ of $\Omega_n$ as the intersection point
    in the shadowed region of the horocycles $C_{2i-2}(n)$ at
    $p_{2i-2}$ and $C_{2i}(n)$ at $p_{2i}$.}
  \label{Fig6}
\end{figure}

\begin{lemma}
  For $n$ big enough, the semi-ideal domain $\O_n$ satisfies the
  conditions of Jenkins-Serrin.
\end{lemma}
\begin{proof}
  We take the horocycles $H_{2i}(m)$, for some fixed $m$, to compute
  the truncated lengths of geodesics in $\O_n$ (see
  Subsection~\ref{subsec}).  By construction, $\a(\O_n)=\b(\O_n)$.
  Suppose there exists an inscribed polygonal domain $\mathcal{P}_n$
  in $\O_n$, $\mathcal{P}_n\neq\O_n$, such that
  $|\Gamma(\mathcal{P}_n)| \leqslant 2\a(\mathcal{P}_n)$ (the case
  $|\Gamma(\mathcal{P}_n)| \leqslant 2\b(\mathcal{P}_n)$ follows
  similarly).  Since $\mathcal{P}_n\neq\O_n$, there is at least an
  interior geodesic arc $\g_1(n)$ in $\partial{\mathcal{P}_n}$ (i.e.
  $\g_1(n)\subset\partial{\mathcal{P}_n}\cap\O_n$). As in the proof of
  Lemma~\ref{lem:JS}, we can assume that
  \[
  \partial{\mathcal{P}_n}= \g_1(n)\cup A_{i_1}(n)\cup\ldots\cup
  \g_s(n)\cup A_{i_s}(n) ,
  \]
  where each $\g_j(n)$ is either an interior geodesic or a $B_i(n)$
  edge in $\partial\O_n$, and at least $\g_1(n)$ is an interior
  geodesic.  In particular, each $\g_j(n)$ joins an even vertex
  $q_{2j}=p_{2i_{j-1}}$ to an odd vertex $q_{2j+1}(n)=p_{2 i_j-1}(n)$.

  Since $\sum_{j=1}^s |\g_j(n)\cap\Gamma(\mathcal{P}_n)|=
  |\Gamma(\mathcal{P}_n)|-\a(\mathcal{P}_n) \leqslant
  \a(\mathcal{P}_n)= s(\ell_n-n+m)$, there must be some interior
  geodesic $\g_j(n)\subset\partial{\mathcal{P}_n}$ such that
  \[
  |\g_j(n)\cap\Gamma(\mathcal{P}_n)|\leqslant \ell_n-n+m=
  |A_{i_j}(n)\cap\Gamma(\mathcal{P}_n)|.
  \]We observe that
  $d(n)=|\g_j(n)\cap\Gamma(\mathcal{P}_n)|-
  |A_{i_j}(n)\cap\Gamma(\mathcal{P}_n)|$, which is non-positive, does
  not depend on $m$.

  Let $\widetilde\Gamma(\mathcal{P}_n)$ be the part of
  $\Gamma(\mathcal{P}_n)$ outside the horocycles $H_{2i-1}(m)$ (we are
  removing a compact part of $\Gamma(\mathcal{P}_n)$), and define
  \[
  \textstyle{\widetilde d(n)=|\g_j(n)\cap\widetilde\Gamma(\mathcal{P}_n)|-
  |A_{i_j}(n)\cap\widetilde\Gamma(\mathcal{P}_n)|
  \stackrel{(\clubsuit)}{\leqslant} d(n)\leqslant 0,}
  \]
  where in {\tiny$(\clubsuit)$} we have used that $A_{i_j}(n)$ is
  contained in a complete geodesic curve finishing at $p_{2 i_j-1}$.

  Now we take the horocycles $H_{2i-1}(m),H_{2i}(m)$ for computing the
  truncated lengths of geodesics of infinite length contained in
  $\O$. Consider the polygonal domain $\mathcal{P}$ inscribed in $\O$
  obtained from $\mathcal{P}_n$ by replacing the vertices which are
  odd vertices $p_{2i-1}(n)$ of $\O_n$ by the corresponding vertices
  $p_{2i-1}$ of $\O$, and denote by $\g_j\subset\partial\mathcal{P}$
  the geodesic from $q_{2j}=p_{2i_{j-1}}$ to $q_{2j+1}=p_{2
    i_{j}-1}$. It is clear that $\mathcal{P}_n$ (resp. $\g_j(n)$,
  $A_{i_j}(n)$) converges to $\mathcal{P}$ (resp. $\g_j$, $A_{i_j}$)
  as $n\to+\infty$. Thus $\widetilde d(n)$ converges to
  $|\g_j\cap\Gamma(\mathcal{P})|-
  |A_{i_j}\cap\Gamma(\mathcal{P})|$. In particular
  $|\g_j\cap\Gamma(\mathcal{P})|\leqslant
  |A_{i_j}\cap\Gamma(\mathcal{P})|$.

  Let $\mathcal{P}_1$ be the component of $\O-\g_j$ containing
  $A_{i_j}$ on its boundary. Clearly, $\mathcal{P}_1$ is a polygonal
  domain inscribed in $\O$, and it holds
  \[
  |\Gamma(\mathcal{P}_1)|- 2\a(\mathcal{P}_1)=
  |\g_j\cap\Gamma(\mathcal{P})|-
  |A_{i_j}\cap\Gamma(\mathcal{P})|\leqslant 0.
  \]
  As this holds for any $m$, we get a contradiction with $\O$ being a
  Jenkins-Serrin domain.
\end{proof}

Using Section~\ref{semi.ideal} and a diagonal argument, we obtain $\O$
as limit of bounded Jenkins-Serrin domains $\widetilde\O_n$, each
$\widetilde\O_n$ with $2k$ edges of length $\ell_n$ and vertices
$p_1(n),\ldots, p_{2k}(n)$ (the odd vertices $p_{2i-1}(n)$ are also
vertices of the semi-ideal domain $\O_n$). Call $\widetilde A_i(n)$
(resp. $\widetilde B_i(n)$) the edge of $\O_n$ whose endpoints are
$p_{2i-1}(n),p_{2i}(n)$ (resp. $p_{2i}(n),p_{2i+1}(n)$).

We now pass to the conjugate surfaces arguing similarly as in
Section~\ref{semi.ideal}. Fix a point $p_0\in\O$. For $n$ large
enough, $p_0\in\widetilde\O_n$.  By Theorem~\ref{th:JS}, there exists
a solution $u$ (resp. $u_n$, for any $n$) to the minimal graph
equation defined over $\O$ (resp. $\widetilde\O_n$), with boundary
values $+\infty$ over $A_i$ (resp. $\widetilde A_i(n)$) and $-\infty$
over $B_i$ (resp. $\widetilde B_i(n)$). Denote by $\Sigma$ (resp.
$\Sigma_n$) the graph surface of $u$ (resp. $u_n$). Up to a vertical
translation we can assume $u(p_0)=u_n(p_0)=0$.  Since $\O$ is a
Jenkins-Serrin domain, there cannot exist divergence lines associated
to the sequence of graphs $u_n$. Hence a subsequence of
$\{\Sigma_n\}_n$ converges uniformly on compact sets to $\Sigma$.

Denote by $\Sigma^*$ (resp. $\Sigma_n^*$) the conjugate surface of
$\Sigma$ (resp. $\Sigma_n$). Remark that $\Sigma^*$ has no boundary,
as $\partial\Sigma=\emptyset$. Since $\Sigma$ and all the $\Sigma_n$
contain a point $p_0$, we can take a fixed point $p_0^*\in\htwor$
(corresponding by conjugation to $p_0$) contained in $\Sigma^*$ and in
$\Sigma_n^*$, for any $n$. Then the surfaces $\Sigma_n^*$ converge to
$\Sigma^*$, as $n\to+\infty$.

By Proposition~\ref{parallel.planes}, $\Sigma_n^*$ is a graph over a
domain $\O_n^*$ and
\[
\partial\Sigma_n^*=\Gamma_1^*(n)\cup\eta_1^*(n)\cup\ldots\cup
\Gamma_k^*(n)\cup\eta_k^*(n) ,
\]
where $\Gamma_1^*(n),\ldots,\Gamma_k^*(n)$
(resp. $\eta_1^*(n),\ldots,\eta_k^*(n)$) are horizontal curvature
lines of symmetry contained in the same horizontal plane, and both
planes are at distance $\ell_n$ from each other.  Assume
\[
\Gamma_1^*(n),\ldots,\Gamma_k^*(n)\subset\{t=-t_n\}\quad \mbox{ and
}\quad \eta_1^*(n),\ldots,\eta_k^*(n)\subset\{t=\ell_n-t_n\} .
\]
Since the distance from $p_0$ to each one of the $2k$ components of
$\partial\Sigma_n$ (which are vertical straight lines over
the vertices of $\widetilde\O_n$) diverges to $+\infty$ as
$n\to+\infty$, the same does for the distance from $p_0^*$ to each
component of $\partial\Sigma_n^*$. Thus both $t_n$ and
$\ell_n-t_n$ diverge to $+\infty$.

Call $\widetilde\Gamma_i^*(n)$ (resp. $\widetilde\eta_i^*(n)$) the
vertical projection of $\Gamma_i^*(n)$ (resp. $\eta_i^*(n)$) over
$\{t=0\}$. Then
$\partial\O_n^*=\widetilde\Gamma_1^*(n)\cup\widetilde\eta_1^*(n)\cup\ldots\cup
\widetilde\Gamma_k^*(n)\cup\widetilde\eta_k^*(n)$, and two consecutive
curves in $\partial\O_n^*$ are asymptotic at $\partial_\infty\htwo$.
As $t_n$ (resp. $\ell_n-t_n$) diverges to $+\infty$ as $n\to+\infty$,
then $u_n|_{\Gamma_i^*}(n)$ (resp. $u_n|_{\eta_i^*(n)}$) diverge to
$-\infty$ (resp. $+\infty$). We deduce that each $\Gamma_i^*(n)$
(resp. $\eta_i^*(n)$) converges to a geodesic $\Gamma_i^*$ (resp.
$\eta_i^*$). Therefore, the domains $\O_n^*$ converge to an ideal
polygonal domain $\O^*$, with
$\partial\O^*=\Gamma_1^*\cup\eta_1^*\cup\ldots\cup\Gamma_k^*\cup\eta_k^*$,
and $\Sigma^*$ is a Jenkins-Serrin graph over $\O^*$ with boundary
values $-\infty$ at each $\Gamma_i^*$ and $+\infty$ at each
$\eta_i^*$.

To finish Theorem~\ref{th:limitsJS}, it remains to prove that, given
$k\geqslant 2$, there exists $(2k-3)$ possible ideal Jenkins-Serrin
domains $\O$, after identifying them by isometries of $\htwo$.  We can
assume that the vertices $p_1,p_2$ are fixed.  Once chosen the
vertices $p_3,\ldots,p_{2k-1}$, the vertex $p_{2k}$ is determined by
$p_{2k-1}$ and $p_1$, as $\O$ is a Jenkins-Serrin domain. Hence we
have $2k-3$ freedom parameters.  This completes
Theorem~\ref{th:limitsJS}.

\begin{remark}[Symmetric case]
  {\rm If $\O$ can be obtained by reflection from an ideal triangle
    $T$ with vertices $p_1,p_2,{\bf 0}$, then $\Sigma^*$ can be got as
    a limit of symmetric Saddle Towers as at item (2) of
    Remark~\ref{rem:symmetry} (when $\lambda=\mu\to+\infty$). Then
    $\Sigma^*$ contains $k$ horizontal geodesics and $k$ vertical
    curvature lines of symmetry, all meeting at the origin with angle
    $\pi/(2k)$. We deduce that $\Sigma^*=\Sigma$; i.e. the
    Jenkins-Serrin graphs over symmetric ideal polygonal domains are
    self-conjugate. }
\end{remark}


\begin{thebibliography}{9}
\bibitem{CR} P. Collin, H. Rosenberg,
  {\it Construction of harmonic diffeomorphisms and minimal graphs},
  preprint.

\bibitem{D} B. Daniel,
  {\it Isometric immersions into ${\mathbb S}^n \times \R $ and
    ${\mathbb H}^n \times \R$ and applications to minimal surfaces},
  Comment. Math. Helv. 82 (2007), no. 1, 87-131. Zbl 1123.53029.


  \bibitem{HR} L. Hauswirth, H. Rosenberg,
  {\it Minimal surfaces of finite total curvature in $\Bbb
              H\times\Bbb R$}, Mat. Contemp., 31 (2006), 65-80.
    MR2385437, Zbl 1144.53323.

  
\bibitem{HST} L. Hauswirth, R. Sa Earp, E. Toubiana,
  {\it Associate and conjugate minimal immersions in ${\mathbb H}^2
    \times \R$}, Tohoku Math. J., (2) 60 (2008), no. 2, 267-286.
    MR2428864, Zbl 1153.53041.

\bibitem{JS} H. Jenkins, J. Serrin,
 {\it Variational problems of minimal surface type II. 
 Boundary value problems for the minimal surfaces equation},
Arch. Rational Mech. Anal., 21 (1966), 321-342. MR0190811,
Zbl 0171.08301.

\bibitem{K1} H. Karcher,
  {\it Embedded minimal surfaces derived from Scherk examples},
  Manuscripta Math., 62, 83-114, 1988. MR0958255, Zbl 0658.53006.

\bibitem{K2} H. Karcher,
 {\it Construction of minimal surfaces}, 
 Surveys in Geometry, 1-96,
 University of Tokyo, 1989 and Lecture Notes n. 12, SFB256, Bonn,
 1989.

\bibitem{K3} H. Karcher, {\it Introduction to conjugate Plateau
    constructions}, Global theory of minimal surfaces, Clay
  Math. Proc., vol. 2, 137--161, Amer. Math. Soc., Providence, RI
  (2005). MR2167258.


\bibitem{LP} H. Lee, J. Pyo. Personal communication.

\bibitem{MRR} L. Mazet, M. M. Rodr\'\i guez, H. Rosenberg, 
  {\it The Dirichlet problem for the minimal surface equation with
    possible infinite boundary data over domains in a Riemannian
    surface}, preprint.

\bibitem{MRT} L. Mazet, M. M. Rodr\'\i guez, M. Traizet, 
{\it Saddle towers with infinitely many ends}, 
Indiana Univ. Math. J.,
  56 (6), 2821-2838 (2007). MR2375703.

\bibitem{NR} B. Nelli, H. Rosenberg,
  {\it Minimal surfaces in ${\mathbb H}^2 \times \R$}, Bull.
  Braz. Math. Soc., 33, (2002), 263-292. MR1940353,
  Zbl 1038.53011.

\bibitem{PT} J. Pérez and M. Traizet,
  {\it The classification of singly periodic minimal surfaces with
    genus zero and Scherk type ends}, Transactions of the AMS, 359
  (3), 965-990 (2007). MR2262839,Zbl 1110.53008.

\bibitem{P} J. Pyo, {\it New complete embedded minimal 
surfaces in $\htwo \times \R$}. Preprint.

\bibitem{S} H.F. Scherk,
  {\it Bemerkungen über die kleinste Fläche innerhalb gegebener
    Grenzen,} J. R. Angew. Math., 13, 185-208, 1935.
  
\end{thebibliography}
\end{document}